\documentclass[leqno,10pt]{amsart}%

\usepackage{amsfonts}
\usepackage[latin1]{inputenc}
\usepackage{amsmath}
\usepackage{amssymb}
\usepackage{amscd}
\usepackage{amsthm}
\usepackage{pifont}
\usepackage{graphicx}%
\usepackage{hyperref}
\providecommand{\U}[1]{\protect\rule{.1in}{.1in}}

\newtheorem{theorem}{Theorem}[section]

\newtheorem{lemma}[theorem]{Lemma}

\newtheorem{proposition}[theorem]{Proposition}
\newtheorem{remark}[theorem]{Remark}

\newtheorem{thm}{Theorem}

\newtheorem{coro}[thm]{Corollary}

\newcommand{\bu}{{\bar u}}
\newcommand{\bv}{{\bar v}}
\newcommand{\bm}{{\bar{M}}}

\newcommand{\ricc}{\operatorname{Ric}}
\newcommand{\mricc}{{}^M \operatorname{Ric}}

\newcommand{\rr}{\mathbb{R}}

\newcommand{\ttt}{\mathbb{T}}
\newcommand{\Hess}{\operatorname{Hess}}

\newcommand{\nsect}{{}^N \operatorname{Sect}}

\newcommand{\vol}{\operatorname{Vol}}

\newcommand{\tu}{\tilde{u}}

\newcommand{\dive}{\operatorname{div}}

\newcommand{\mmetr}[2]{g_{M}\left( #1,#2\right)}

\newcommand{\nmetr}[2]{g_{N}\left( #1,#2\right)}
\newcommand{\hsmetr}[2]{\left\langle #1,#2\right\rangle_{HS}}
\newcommand{\Capp}[1]{\operatorname{Cap}_{#1}}

\begin{document}

\title{Topology of steady and expanding gradient Ricci solitons via $f$-harmonic maps}
\date{\today}
\author{Michele Rimoldi}
\address{Dipartimento di Scienza e Alta Tecnologia,
Universit\`a degli Studi dell'Insubria, \\
via Valleggio 11, 
I-22100 Como, ITALY}
\email{michele.rimoldi@gmail.com}

\author{Giona Veronelli}
\address{
Universit\'e Paris 13, Sorbonne Paris Cit\'e, LAGA, CNRS ( UMR 7539)\\
99, avenue Jean-Baptiste Cl\'ement F-93430 Villetaneuse - FRANCE }
\email{giona.veronelli@gmail.com}


\begin{abstract}In this paper we give some results on the topology of manifolds with $\infty$--Bakry--\'Emery Ricci tensor bounded below, and in particular of steady and expanding gradient Ricci solitons. To this aim we clarify and further develop the theory of $f$--harmonic maps from non--compact manifolds into non--positively curved manifolds. Notably, we prove existence and vanishing results which generalize to the weighted setting part of Schoen and Yau's theory of harmonic maps.
\end{abstract}
\maketitle

\section{Introduction and main results}

Let $(M^{m},g_{M})$ and $(N^n,g_{N})$ be complete Riemannian manifolds, $\dim M=m\geq 2$, $\dim N=n$. Let $f:M\to\rr$ be a smooth function. A map $u:M\to N$ is said to be (weakly) $f$--harmonic if $u|_{\Omega}$ is a critical point of the $f$--energy
\[
E_f(u)=\frac12\int_{\Omega}e^{-f}|du|_{HS}^2 dV_M
\]
for every compact domain $\Omega\in M$. Here $|\cdot|_{HS}$ denotes the Hilbert--Schmidt norm on the set $T^{\ast}M\otimes u^{-1}TN$ of the vector--valued $1$--forms along the map $u$ and $dV_M$ stands for the canonical Riemannian volume form on $M$. When $u$ is $C^2$--regular, the Euler--Lagrange equation for the energy functional $E_f$ is the $f$--harmonic maps equation \cite{L,Co}
\[
\tau_fu:= e^{f}\dive (e^{-f}du) = \tau u - i_{\nabla f}du = 0,
\]
where $\tau u = \dive du$ is the standard tension field of $u$, so that $\tau_f u$ is naturally named $f$--tension field of $u$. Here $i$ denotes the interior product on $1$--forms, i.e. $i_{\nabla f}du = du (\nabla f)$, while $-\dive$ stands for the formal adjoint of the
exterior differential $d$ with respect to the standard $L^{2}$ inner product on vector--valued $1$--forms. The study of $f$--harmonic maps began with A. Lichnerowicz in 1969 \cite{L} and J. Eells and L. Lemaire in 1977, \cite{EL-Bull}, but apparently this subject has been very poorly investigated later. Let us just recall the recent works of N. Course \cite{Co,Co2}, especially about $f$--harmonic flow on surfaces, and Y.--L. Ou \cite{Ou} about $f$--harmonic morphisms. 
A more general class of maps, named pseudo-harmonic maps, has been analyzed by G. Kokarev in \cite{Ko-London}. After the first drafting of our work, two new papers have appeared. In the first, by G. Wang and D. Xu, \cite{WX-IJM}, there are proved some vanishing results which partially recover our Theorem \ref{th_main}. In the second one, Q. Chen, J. Jost and H. Qiu presented some connected results on $V$--harmonic maps, \cite{CJQ-AGAG}. Namely, the authors considered there solutions to the $V$--harmonic maps equation $\tau_Vu:= \tau u - i_{V}du = 0$, where $V$ is a vector field non--necessarily of gradient type.

Note that the $f$--harmonicity of a map $u$ defined on a Riemannian manifold $M$ is equivalent to the harmonicity of $u$ on some related manifolds, see Proposition \ref{prop_conf} and Proposition \ref{prop_warp} below. For instance, if $m\geq 3$ then $u:\left(M^m,g_{M}\right)\to N$ is $f$--harmonic if and only if $u:\left(M^m,e^{-\frac{2f}{m-2}}g_{M}\right)\to N$ is harmonic.  Nevertheless, the interaction of $f$--harmonicity with curvature conditions looks promising and justifies the study of $f$-harmonic maps in order to deduce information on smooth metric measure spaces and gradient Ricci solitons. Some first geometric consequences of this approach will be pointed out in the Section \ref{main}, where new interesting results on the topology of steady and expanding gradient Ricci solitons are deduced. For this reason, notation used in this paper is that of smooth metric measure spaces, e.g. \cite{WW-JDG,WW2} and not the one introduced so far for $f$--harmonic maps, where often $e^{-f}$ is replaced by $f$ \cite{L,Co,Ou}.

A smooth metric measure space, also known in the literature as a weighted manifold, is a Riemannian manifold $\left(M^m,g_{M}\right)$ endowed with a weighted volume form $e^{-f}dV_M$, for some smooth function $f:M\to \rr$.
Associated to a smooth metric measure space $(M^m, g_{M}, e^{-f}dV_M)$ there is also a natural divergence form second order diffusion operator: the $f$--Laplacian. This is defined on $u\in C^2(M)$ by $\Delta_fu=e^{f}\dive(e^{-f}\nabla u)=\Delta u-\mmetr{\nabla u}{\nabla f}$ and we can note that, for real--valued functions, $\Delta_fu=\tau_fu$. A natural question that arises in the setting of smooth metric measure spaces is what is the right concept of curvature on these spaces. Actually there is not a canonical choice. Good choices are those that reveal interplays with metric and topological properties of the space, see e.g. \cite{Mo-Kodai, WW-JDG, WW2}. We are interested in the $\infty$--Bakry--\'Emery Ricci tensor ${}^M\ricc_f={}^M\ricc+\Hess f$, which was first introduced by A. Lichnerowicz  in \cite{Li} and later by D. Bakry and M. \'Emery in \cite{BE}.
Recently it has been found that this curvature tensor is strictly related with geometric objects whose importance is outstanding in mathemathics. Imposing the constancy of ${}^{M}\ricc_f$, one introduces on the manifold an additional structure which goes under the name of gradient Ricci soliton structure. 
Namely, recall that, given a Riemannian manifold $(M, g_{M})$, a Ricci soliton structure on $M$ is the choice of a smooth vector field $X$ (if any) satisfying the soliton equation $\ricc+\frac{1}{2}\mathcal{L}_Xg_{M}=\lambda g_{M}$, for some $\lambda\in\mathbb{R}$. The Ricci soliton $(M, g_{M}, X)$ is said to be shrinking, steady or expanding according to whether $\lambda>0$, $\lambda=0$ or $\lambda<0$. In the special case where $X=\nabla f$ for some smooth function $f:M\to\mathbb{R}$, we have that 
$
\ricc_f=g_{M},
$
and we say that $(M, g_{M}, \nabla f)$ is a gradient Ricci soliton with potential $f$.
The importance of gradient Ricci solitons is due the fact that they correspond to ``self--similar'' solutions to Hamilton's Ricci flow and often arise as limits of dilations of singularities developed along the flow.

Mimicking the theory developed by J. Eells and J. H. Sampson \cite{ES} and P. Hartman \cite{Ha}, in the seminal paper \cite{L} Lichnerowicz use $f$--harmonic maps to deduce topological information on compact $(m\geq3)$--dimensional smooth metric measure spaces satisfying ${}^M\ricc_f\geq 0$. In this paper, following the theory for harmonic maps developed by R. Schoen and S.--T. Yau \cite{SY-CH}, we extend the approach of Lichnerowicz in order to study the topology of complete manifolds of dimension $m\geq 2$ with non--negative Bakry--\'Emery Ricci tensor ${}^M\ricc_f\geq 0$ and of expanding gradient Ricci solitons with a suitable control on the scalar curvature. This problem is particularly interesting since very poor information on the topology of this class of manifolds is known, see Section \ref{main}.

Concerning $f$--harmonic maps, we obtain the following general result.

\begin{thm}\label{th_main}
Let $M$ be a complete non--compact Riemannian manifold and $N$ a compact Riemannian manifold with $\nsect\leq0$. Let $f\in C^\infty(M)$ and consider a continuous map $u:M\to N$ with finite $f$--energy $E_f(u)< +\infty$.
\begin{itemize}
\item[(\textbf{I})] Assume that $\mricc_f\geq 0$ and that at least one of the following assumption is satisfied
\begin{itemize}
\item[(a)] there exists a constant $C>0$ such that $|f|\leq C$;
\item[(b)] $f$ is convex and the set of its critical points is unbounded;
\item[(c)] $\vol_f(M):=\int_Me^{-f}dV_M=+\infty$;
\item[(d)] there is a point $q_0\in M$ such that $\mricc_f|_{q_0}>0$;
\item[(e)] there is a point $q_1\in M$ such that $\mricc(X,X)|_{q_1}\neq 0$ for all $0\neq X\in T_{q_1}M$.
\end{itemize}
Then $u$ is homotopic to a constant. 
\item[(\textbf{II})] Assume that $\mricc_f\geq -k^2(x)$ for some $0\leq k\in C^\infty(M)$, $k\not\equiv 0$, such that
\begin{itemize}
\item[(f)] $\lambda_1(-\Delta_f-Hk^2)\geq 0$ for some $H>1$.
\end{itemize}
Then $u$ is homotopic to a constant.
\item[(\textbf{III})] If $\mricc_f\geq 0$ and $\nsect<0$, then $u$ is homotopic either to a constant or to a totally geodesic map whose image is contained in a geodesic of $N$.
\end{itemize}
\end{thm}

\begin{remark}\label{rmk_rai}{\rm 
By Rayleigh characterization, the spectral assumption 
\begin{equation}\label{f-spect}
\lambda_1(-\Delta_f-Hk^2)\geq 0
\end{equation} 
is equivalent to ask that 
\begin{equation}\label{ray}
\int_M|\nabla\varphi|^2e^{-f}dV_M-\int_MHk^2\varphi^2e^{-f}dV_M\geq 0
\end{equation}
for all $\varphi\in C^{\infty}_c(M)$. Then, in case $\mricc_f\geq 0$, we are not assuming \eqref{f-spect}, since it is trivially satisfied.

The non--weighted $k\not\equiv 0$ version of Theorem \ref{th_main} is due to \cite{PRS-JFA05}. See also \cite{PV}, where the lower bound on $H$ is improved, and Remark \ref{rmk_kato} below for a comment on the lower bound for $H$.
}\end{remark}

\begin{remark}{\rm
At least in case $k$ is bounded, the assumptions of Theorem \ref{th_main} can be weakened by assuming $H=H(x)$ to be a (nonconstant) smooth function $H(x)>1$. This improvement can be pretty useful in some particular situation (see Theorem \ref{th_exp} for an application and its proof in Section \ref{main} for an idea of the proof of this improvement.)
}\end{remark}

Theorem \ref{th_main} has the following topological implications.

\begin{coro}\label{coro_main}
Let $M^m$ be a complete non--compact $m$--dimensional Riemannian manifold. Let $D\subset M$ be a compact domain in $M$ with smooth, simply connected boundary. 
Then,
\begin{itemize}
\item[(i)] under the assumptions in (\textbf{I}) or (\textbf{II}) of Theorem \ref{th_main}, there is no non--trivial homomorphism of $\pi_1(D)$ into the fundamental group of a
compact manifold with non--positive sectional curvature;
\item[(ii)]if $\ricc_f\geq0$, each homomorphism of $\pi_1(D)$ into the fundamental group of a
compact manifold $N$ with strictly negative sectional curvature $\nsect<0$ is either trivial or maps all $\pi_1(D)$ into a cyclic subgroup of $\pi_1(N)$.
\end{itemize}
\end{coro}
\begin{remark}
\rm{Part (ii) of Corollary \ref{coro_main} was pointed out by Schoen and Yau in the non--weighted case \cite{SY-CH}.

Note that Corollary \ref{coro_main} (ii) and Corollary \ref{coro_main} (i) in the assumptions (\textbf{I}.d), (\textbf{I}.e) or (\textbf{II}) hold not asking $f$ to be neither bounded nor convex. This is a reason of interest in the approach we propose, since it permits to deal with cases for which the techniques introduced so far seem to be unapplicable. See Proposition \ref{th_knowntopo} below.
}\end{remark}

The $k\equiv0$ case of Corollary \ref{coro_main} directly applies to study the topology of complete steady gradient Ricci solitons, where few information is known, as discussed in Section \ref{main} below. 

\begin{coro}\label{coro_steady}
Let $(M^m, g_{M}, \nabla f)$ be a complete
non--trivial steady gradient Ricci soliton then the conclusion of
Corollary B (ii) holds. Moreover, if we further assume that one of the
conditions (b), (c), (e) is satisfied, then the conclusion of
Corollary B (i) also holds.
\end{coro}

\begin{remark}\label{Wu..}
\rm{Note that assumption (a) of Corollary \ref{coro_main} cannot hold for non--trivial (i.e. with $f$ non-constant) steady gradient Ricci solitons. Indeed, it has been shown by O. Munteanu and N. Sesum, \cite{MuntSes}, and indipendently by P. Wu, \cite{Wu}, that the potential function $f$ must be unbounded in this case. Moreover it would be interesting to understand whether assumption (c) is always satisfied by steady gradient Ricci solitons.}
\end{remark}

The flexibility given by the spectral assumption permits to deduce information also concerning expanding gradient Ricci solitons. In particular, according to an idea in \cite{MW2}, there exist situations in which the spectral assumption is implied by more geometrical curvature condition.

\begin{thm}\label{th_exp}
Let $(M^m, g_{M}, \nabla f)$ be a complete non--trivial expanding gradient Ricci soliton with scalar curvature ${}^MS > (m-1)\lambda$.  Let $D\subset M$ be a compact domain in $M$ with smooth, simply connected boundary. Then, there is no non--trivial homomorphism of $\pi_1(D)$ into the fundamental group of a compact manifold with non--positive sectional curvature. 
\end{thm}

The paper is organized as follows. In Section \ref{sec_exis} we discuss the relations between the weighted and the non--weighted setting, and in particular under which conditions $f$--harmonic maps on a manifold $M$ can be interpreted as harmonic maps on some related manifold. This permits to deduce the existence of a smooth $f$--harmonic representative in the homotopy class of a given finite $f$--energy map, provided $N$ is compact and $\nsect\leq 0$. Moreover, some information about the uniqueness of the representative is given. In Section \ref{sec_vanishing} we prove a vanishing result for finite $f$--energy $f$--harmonic maps, in the special (easier) case $k\equiv 0$. Then, in Section \ref{SpectSect} we discuss the spectral assumption \eqref{f-spect} and provide the changes needed to complete the proof of Theorem \ref{th_main}. Finally, in Section \ref{main} we deduce geometrical applications to smooth metric measure spaces and gradient Ricci solitons.

\section{Existence and relations between the weighted and the non--weighted setting}\label{sec_exis}

This section aims to give the proof of the following existence result. To obtain this, we will formalize some useful links between harmonic and $f$--harmonic maps.

\begin{theorem}\label{f-burstall}
Let $M^m$ and $N^n$ be Riemannian manifolds, $m\geq 2$. Assume $N$ is compact and $\nsect\leq 0$. Then any homotopy class of maps from $M$ into $N$ containing a continuous map of finite $f$--energy contains a smooth $f$--harmonic map minimizing the $f$--energy in the
homotopy class.
\end{theorem}

\begin{remark}{\rm 
When $M$ is compact and $m\geq 3$, Theorem \ref{f-burstall} is due to \cite{EL-Bull}, p.48.
}\end{remark}

\begin{remark}{\rm
Instead of $\nsect\leq 0$ it would be enough for $N$ to be a $K(\pi, 1)$--space and to admit no non--trivial minimizing tangent maps or $r$--spheres
for $2\leq r \leq m -1$. For instance this is the case if the universal cover of $N$ supports a strictly convex exhaustion function; see also \cite{burstall-london, We}.
}\end{remark}

In dimension $m>2$ the proof is pretty easy. First, by straightforward computations, \cite{L}, one can prove that 
\begin{proposition}\label{prop_conf} 
A map $u:(M^m,g_{M})\to (N^n,g_{N})$, $m\geq 3$, is $f$--harmonic if and only if $u:(M^m,e^{-\frac{2f}{m-2}}g_{M})\to(N^n,g_{N})$ is a harmonic map.
\end{proposition}

In \cite{burstall-london}, F. Burstall proved that if $\nsect\leq 0$, then in the homotopy class of each finite energy map there exists a smooth harmonic representative which minimizes the energy in the homotopy class. As a matter of fact, even if the result is there stated for complete manifolds, the proof does not require the underlying manifold $M$ to be complete. Accordingly, from Proposition \ref{prop_conf} and Theorem 5.2 in \cite{burstall-london} we get the validity of Theorem \ref{f-burstall} when $m>2$. On the other hand, on $2$-dimensional manifolds, it is easily seen that ($f$-)energy is conformally invariant so that this approach does not work. Thus, we follow a different strategy which permits to obtain the result in all dimensions $m\geq 2$. Namely, it was suggested in Section 1.2 of \cite{Co} that $f$--harmonicity on $M$ is expected to correspond to harmonicity on some higher dimensional warped product manifold. We will make this fact explicit.

We consider the warped product $\bm^n=M\times_h\ttt$, where $h:=e^{-f}$ and $\ttt=\ttt^1=\rr/\mathbb{Z}$, so that $\vol(\ttt)=1$. Here and on, each point in $\bm$ is individuated by its projections $x$ on $M$ and $t$ on $\ttt$. Moreover we recall that the metric on $\bm$ is given by $g_{\bm}(x,t)=g_{M}(x)+h^2(x)dt^2$. Throughout the following proofs, $\{E_i\}_{i=1}^n$ is a local orthonormal frame at $(x,t)\in\bm$ such that $\{E_j\}_{j=1}^{n-1}$ is a local orthonormal frame at $x\in M$ and $E_n=h^{-1}\frac{\partial}{\partial t}\in T_t\ttt$.

Even if this is not necessarily used in the proof of Theorem \ref{f-burstall} given below, we start pointing out the explicit relation, of its own interest, between the $f$--tension field on $M$ and the tension field on $\bar M$.
\begin{proposition}\label{prop_warp}
Given a $C^2$ map $v:M\to N$, define the $C^2$ map $\bar v:\bm\to N$ as $\bv(x,t):=v(x)$ for all $(x,t)\in\bm$. Then
\[
\tau\bv(x,t)=\tau_fv(x).
\]
In particular, $v$ is $f$--harmonic if and only if $\bv$ is harmonic.
\end{proposition}
\begin{proof}
Following the rules of covariant derivatives on warped products, see \cite{On} p. 206, we can compute
\[
\left({}^\bm\nabla_{E_j}E_j\right)(x,t) = \left(\left({}^M\nabla_{E_j}E_j\right)(x),0\right)
\]
for $j=1,\dots,n-1$, and 
\[
\left({}^\bm\nabla_{E_n}E_n\right)(x,t) = \left(-\frac{{}^M\nabla h(x)}{h},\left({}^\ttt\nabla_{E_n}E_n\right)(t)\right)=\left({}^M\nabla f(x),0\right).
\]
Moreover, by definition of $\bv$ we have
\[
\left({}^N\nabla_{d\bv(E_j)}d\bv(E_j)\right)(x,t) = \left({}^N\nabla_{dv(E_j)}dv(E_j)\right)(x)
\]
for $j=1,\dots,n-1$, and 
\[
\left({}^N\nabla_{d\bv(E_n)}d\bv(E_n)\right)(x,t) = 0.
\]
Then, we get
\begin{align*}
\Hess \bv|_{(x,t)}(E_j,E_j) &= \left({}^N\nabla_{d\bv(E_j)}d\bv(E_j) - d\bv ({}^\bm\nabla_{E_j}E_j)\right)(x,t) \\
&=  \left({}^N\nabla_{dv(E_j)}dv(E_j) - dv ({}^M\nabla_{E_j}E_j)\right)(x)
\end{align*}
for $j=1,\dots,n-1$, and 
\begin{align*}
\Hess \bv|_{(x,t)}(E_n,E_n) &= \left({}^N\nabla_{d\bv(E_n)}d\bv(E_n) - d\bv ({}^\bm\nabla_{E_n}E_n)\right)(x,t) \\
&=- dv \left(({}^M\nabla f)(x)\right).
\end{align*}
We thus obtain
\begin{align*}
\tau\bv(x,t) &= \sum_{i=1}^n\Hess \bv|_{(x,t)}(E_i,E_i) \\
&= \sum_{j=1}^{n-1}\Hess v|_{x}(E_j,E_j)- dv \left(({}^M\nabla f)(x)\right)\\
&=\tau v(x) - dv \left(({}^M\nabla f)(x)\right) = \tau_fv(x).
\end{align*}
\end{proof}
We have shown that any $f$--harmonic map on $M$ is harmonic when trivially extended to $\bm = M\times_{e^{-f}} \ttt^1$. Clearly, the converse does not hold, since in general a harmonic map from $\bm$ to $N$ does depend on $t$. Indeed, it suffices to consider the example given by $M=\rr$, $f\equiv 0$ and $N=\ttt^2=\ttt\times\ttt$, and it is easily seen that the locally isometric covering map $P: \rr\times\ttt=\bar M \to N=\ttt^2$ is harmonic, but it is not of the form $P(x,t)\equiv P(x,T)$ for any $T\in\ttt$. Nevertheless, as the following proof shows, this converse property holds true for some specific harmonic maps on $\bm$ which minimize energy in their homotopy class.

\begin{proof}[Proof (of Theorem \ref{f-burstall}).]
\textbf{Step a.} Let $v:M\to N$ be a continuous map satisfying $E_f(v)<+\infty$ and define the continuous function $\bar v:\bm\to N$ as $\bv(x,t):=v(x)$. 
We note that
\begin{align*}
|d\bv|^2_{HS(\bm,N)}&=\sum_{i=1}^n\nmetr{d\bv(E_i)}{d\bv(E_i)}\\&=\sum_{j=1}^{n-1}\nmetr{dv(E_j)}{dv(E_j)}=|dv|^2_{HS(M,N)},
\end{align*}
and the latter relation holds in the weak sense for $v,\bv\in C^0$. Since $v$ has finite $f$--energy and $\vol(\ttt)=1$, we can apply Fubini's theorem to get
\begin{align}\label{rel_en}
E^\bm(\bar v)&=\int_\bm|d\bv|^2_{HS(\bm,N)}dV_\bm  
= \int_M e^{-f(x)}\int_{\ttt}|dv|^2_{HS(M,N)}dtdV_M\\ 
&= \int_M e^{-f(x)}|dv|^2_{HS(M,N)}dV_M = E_f^M(v) \nonumber
\end{align}
So $\bv$ is a continuous finite energy map from $\bm$ to a compact manifold $N$ with $\nsect\leq 0$, and according to \cite{burstall-london} we know that there exists a smooth harmonic map $\bu:\bm\to N$ which minimizes the energy in its homotopy class. 

\textbf{Step b.} In this step, we prove that we can choose $\bu$ such that it has the form $\bu(x,t)=u(x)$ for some smooth map $u:M\to N$. In fact, since $E^{\bm}(\bu)<+\infty$, we can apply Fubini's theorem to get
\begin{align*}
\infty>E^\bm(\bu)&=\int_Me^{-f(x)}\int_\ttt|d\bu|^2_{HS(\bm,N)}(x,t)dtdV_M(x)\\&=\int_\ttt\left(\int_Me^{-f(x)}|d\bu|^2_{HS(\bm,N)}(x,t)dV_M(x)\right)dt,
\end{align*}
and we can choose $T\in\ttt$ such that
\begin{align}\label{tT}
&\int_\ttt\left(\int_Me^{-f(x)}|d\bu|^2_{HS(\bm,N)}(x,T)dV_M(x)\right)dt\\
&\leq\int_\ttt\left(\int_Me^{-f(x)}|d\bu|^2_{HS(\bm,N)}(x,t)dV_M(x)\right)dt.\nonumber
\end{align}
We can define a further smooth map $\tu:\bm\to N$ as $\tu(x,t):=\bu(x,T)$ for all $(x,t)\in\bm$. Since $d\tu(E_n)=0$, we have that
\begin{align*}
|d\tu|^2_{HS(\bm,N)}(x,t)
&=|d\tu|^2_{HS(\bm,N)}(x,T)
=\sum_{j=1}^{n-1}\nmetr{d\tu(E_j)}{d\tu(E_j)}(x,T)\\
&=\sum_{j=1}^{n-1}\nmetr{d\bu(E_j)}{d\bu(E_j)}(x,T)\\
&\leq \sum_{i=1}^{n}\nmetr{d\bu(E_i)}{d\bu(E_i)}(x,T)\\
&=|d\bu|^2_{HS(\bm,N)}(x,T)
\end{align*}
for all $(x,t)\in\bm$. From \eqref{tT} we thus obtain
\begin{align}\label{ener}
E^{\bm}(\tu)&=\int_\ttt\left(\int_Me^{-f(x)}|d\tu|^2_{HS(\bm,N)}(x,t)dV_M(x)\right)dt\\
&\leq \int_\ttt\left(\int_Me^{-f(x)}|d\bu|^2_{HS(\bm,N)}(x,T)dV_M(x)\right)dt\nonumber\\
&\leq \int_\ttt\left(\int_Me^{-f(x)}|d\bu|^2_{HS(\bm,N)}(x,t)dV_M(x)\right)dt=E^{\bm}(\bu)\nonumber
\end{align}
Now we prove that $\tu$ is homotopic to $\bu$. To this end, we recall that since $N$ is a $K(\pi,1)$--space, the homotopy class of a map $w:\bm\to N$ is completely characterized by the action on $\pi_1(\bar M)$ of the homeomorphism induced by $w$,
$$w_\sharp:\pi_1(\bm,(x,t))\to\pi_1(N,w(x,t)).$$
We have that $\pi_1(\bm,(x,t))\cong\pi_1(M,x)\times\pi_1(\ttt,t)$. Since $\bu$ is homotopic to $\bv$ and, by construction, $\bv_\sharp(\gamma)=id\in\pi_1(N)$ for all $\gamma\in\pi_1(\ttt,t)<\pi_1(\bm,(x,t))$, we have that $\bu_\sharp(\gamma)=id\in\pi_1(N)$ for all $\gamma\in\pi_1(\ttt,t)$. On the other hand, by definition of $\tu$ we have that $\tu_\sharp(\gamma)=id\in\pi_1(N)$ for all $\gamma\in\pi_1(\ttt,t)$ and $\tu_\sharp|_{\pi_1(M,x)}$ is conjugated to $\bu_\sharp|_{\pi_1(M,x)}$. Accordingly, $\tu_\sharp$ is conjugated to $\bu_\sharp$, so that $\tu$ is homotopic to $\bu$. Since we have proved above that $\bu$ minimizes the energy in its homotopy class, this together with \eqref{ener} implies that $E^\bm(\tu)=E^\bm(\bu)$. Hence $\tu$ is a smooth minimizer of the energy in its homotopy class, it is thus harmonic and it has the aimed form.

\textbf{Step c.} We have shown that there exists a smooth map $u:M\to N$ such that $\tu(x,t)=u(x)=\bu(x,T)$ for all $(x,t)\in\bm$. This map $u$ is a minimizer of the $f$--energy in its homotopy class. In fact, by contradiction suppose there exists another map $u_0:M\to N$ homotopic to $u$ with $E_f^M(u_0)<E_f^M(u)$. Then we can define $\tu_0:\bm\to N$ as $\tu_0(x,t)=u_0(x)$ for all $(x,t)\in \bm$ and reasoning as above we would get that $\tu_0$ is homotopic to $\tu$ and $E^\bm(\tu_0)<E^\bm(\tu)$. This would give the contradiction.

Hence $u$ is a smooth minimizer of the $f$--energy in its homotopy class, and it is therefore an $f$--harmonic map.
\end{proof}

To conclude this section, we would like to say some words to characterize the $f$--harmonic representative in homotopy class. In fact, thanks to the relations between $M$ and $\bar M$ pointed out above, it is pretty easy to see how to generalize the uniqueness results obtained in the harmonic setting by Schoen and Yau \cite{SY-Topo} when $M$ has finite volume (see also Remark 4 in \cite{PRS-MathZ} for the improved parabolic version). We recall that a manifold $M$ is said to be $f$--parabolic if for some (hence every) compact set $K\subset M$ with non--empty interior the $f$--capacity of $K$ is null, i.e.
\[
\Capp f(K)=\inf \left\{E_f(\varphi)\ :\ \varphi\in C_c^{\infty}(M),\ \left.\varphi\right|_K\geq1 \right\}=0.
\] 
There exist several equivalent definitions of $f$--parabolicity. To our purpose, let us just recall that $M$ is $f$--parabolic if and only if every bounded $f$--subharmonic function is necessarily constant. This equivalence can be easily proved by adapting to the weighted setting the standard arguments used in the non--weighted case; see e.g. \cite{Gr3}. Now, consider a compact set $K\subset M$ with non--empty interior and define the compact set $\bar K=K\times\ttt\subset \bar M$. First, suppose $\Capp f(K)=0$ in $M$. Then by definition of $f$--capacity and by relation \eqref{rel_en}, it is clear that $\Capp{}(\bar K)=0$ in $\bar M$. On the other hand, suppose $\bar M$ is parabolic. Let $\psi:M\to\rr$ be a bounded $f$--subharmonic function. Defining $\bar\psi:\bar M\to \rr$ as in Proposition \ref{prop_warp}, we have that $\bar\psi$ is a bounded subharmonic function on $\bar M$. Hence $\bar\psi$ and in turn $\psi$ are constant. This proves the following
\begin{lemma}
Let $M$ be a complete Riemannian manifold and $f\in C^\infty(M)$. Then $M$ is $f$--parabolic if and only if $\bar M = M\times_{e^{-f}}\ttt$ is parabolic.
\end{lemma}
Now, suppose $M$ is $f$--parabolic, $N$ is non--positively curved and we are given two homotopic smooth $f$--harmonic maps $u,v:M\to N$ with finite $f$--energy. Defining maps $\bar u, \bar v:\bar M\to N$ as above we have that $\bar u$ and $\bar v$ are smooth finite--energy harmonic maps on the parabolic manifold $\bar M$, and they are homotopic since they are constantly extended on the fiber $\ttt$. Then we can apply Theorems 1 and 2 in \cite{SY-Topo} to $\bar u,\bar v$ to get

\begin{theorem}\label{Ith_SYf}
Let $M$ and $N$ be complete Riemannian manifolds and assume that $M$ is $f$--parabolic,
\begin{itemize}
\item[i)] Let $u:M\to N$ be a smooth $f$--harmonic map of finite $f$--energy. If $\nsect <0$, there's no other $f$--harmonic map of finite $f$--energy homotopic to $u$ unless $u(M)$ is contained in a geodesic of $N$.
\item[ii)] If $\nsect\leq 0$ and $u,v:M\to N$ are homotopic smooth $f$--harmonic maps of 
finite $f$--energy, then there is a continuous one--parameter family of maps $u_s:M\to N$ with $u_0=u$ and $u_1=v$ such that every $u_s$ is a $f$--harmonic map of constant $f$--energy (independent of $s$) and for each $q\in M$ the curve $s\mapsto u_s(q)$, $s\in[0,1]$, is a constant (independent of $q$) speed parametrization of a geodesic. 
\end{itemize}
\end{theorem}  

The fact that the harmonic maps $\bar u_s:\bar M \to N$ obtained in the proof of Theorem 2 in \cite{SY-Topo} have the form $\bu_s(x,t):=u_s(x)$ for all $(x,t)\in\bm$ is not a direct consequence of the statement of Theorem 2 in \cite{SY-Topo}, but it can be easily deduced by the construction given in Schoen and Yau's proof.

 \section{Vanishing results}\label{sec_vanishing}

A classical approach in harmonic maps theory leads to obtain vanishing results for the harmonic representative, imposing some additional assumptions on the curvature of $M$. In order to extend this part of the theory to our setting, we need a Bochner formula for $f$--harmonic maps. Such a formula is well known for functions \cite{BE}, while for maps is essentialy contained in \cite{L}, Section 13. Since apparently no explicit version is given in the literature, we give a detailed proof in the following. 
\begin{proposition}\label{prop_f-bochner}
Let $v : (M^m,g_{M}) \to (N^n, g_{N})$ be a $C^2$ map. Then
\begin{align}\label{f-bochner}
\frac12\Delta_f|dv|^2 
&= |Ddv|^2 
+ \hsmetr{dv}{d\tau_fv} 
+ \sum_{i=1}^m\nmetr{dv(\mricc_f(E_i,\cdot)^{\sharp})}{dv (E_i)}\\
&- \sum_{i,j=1}^m\nmetr{{}^N\operatorname{Riem}(dv (E_i),dv (E_j))dv (E_j)}{dv (E_i)} \nonumber, 
\end{align}
where $\{E_i\}_{i=1}^m$ is some chosen orthonormal frame on $M$.
\end{proposition}

\begin{proof}
We start recalling the standard Bochner formula for the smooth map $v$, \cite{EL},
\begin{align}\label{bochner}
\frac12\Delta|dv|^2 
&= |Ddv|^2 
+ \hsmetr{dv}{d\tau v} 
+ \sum_{i=1}^m\nmetr{dv(\mricc(E_i,\cdot)^{\sharp})}{dv (E_i)}\\
&- \sum_{i,j=1}^m\nmetr{{}^N\operatorname{Riem}(dv (E_i),dv (E_j))dv (E_j)}{dv (E_i)} \nonumber.
\end{align}
Inserting
\[
\frac12\Delta_f|dv|_{HS}^2=\frac12\Delta|dv|_{HS}^2- \frac12\nabla f(|dv|_{HS}^2)
\]
and
\[
\hsmetr{dv}{d\tau_f v}=\hsmetr{dv}{d\tau v}- \hsmetr{dv}{d(i_{\nabla f} dv)}
\]
in \eqref{bochner} we get
\begin{align*}
\frac12\Delta_f|dv|^2 
&= |Ddv|^2 
+ \hsmetr{dv}{d\tau_f v} 
+ \sum_{i=1}^m\nmetr{dv(\mricc(E_i,\cdot)^{\sharp})}{dv (E_i)}\\
& + \hsmetr{dv}{d(i_{\nabla f} dv)} - \frac12\nabla f(|dv|_{HS}^2)
\\
&- \sum_{i,j=1}^m\nmetr{{}^N\operatorname{Riem}(dv (E_i),dv (E_j))dv (E_j)}{dv (E_i)} \nonumber.
\end{align*}
Hence \eqref{f-bochner} is proved once we show that 
\begin{equation}\label{hess}
\hsmetr{dv}{d(i_{\nabla f} dv)} - \frac12\nabla f(|dv|_{HS}^2) = \hsmetr{dv(\nabla_{(\cdot)}\nabla f)}{dv}. 
\end{equation}
Let $\{x^a\}_{a=1}^m$ be a local coordinate chart on $M$ at $q\in M$ and $\{\theta^A\}_{A=1}^n$ and $\{E_A\}_{A=1}^n$ orthonormal coframe and dual frame on $N$ at $v(q)$ respectively. Moreover denote the components of the metric on $M$ as $\mmetr{\frac{\partial}{\partial x^a}}{\frac{\partial}{\partial x^b}}=:g_{ab}$.
We will write in coordinates
\[
dv = v^A_a E_A\otimes dx^a,\qquad \textrm{and}\qquad \nabla f = f^a\frac{\partial}{\partial x^a}.
\]
Then
\[
i_{\nabla f}dv = v^A_af^a E_A,
\]
which gives
\[
d(i_{\nabla f}dv) = \left(v^A_{ab}f^b + v^A_bf^b_{\ a}\right) E_A \otimes dx^a
\]
and
\begin{equation}\label{hess_1}
\hsmetr{dv}{d(i_{\nabla f} dv)} = \sum_{A=1}^n v^A_dv^A_{ab}f^ag^{db}+ \sum_{A=1}^n v^A_dv^A_{a}f^a_{\ b}g^{db}. 
\end{equation}
Moreover, 
\[
\nabla_{(\cdot)}\nabla f = \left(f^a_{\ b}+ \Gamma^a_{bd}f^d\right)\frac{\partial}{\partial x^a}\otimes dx^b,
\]
from which
\[
dv(\nabla_{(\cdot)}\nabla f )= \left(v^A_af^a_{\ b}+ v^A_a\Gamma^a_{bd}f^d\right) E_A \otimes dx^b,
\]
and
\begin{equation}\label{hess_2}
\hsmetr{dv(\nabla_{(\cdot)}\nabla f )}{dv}= \sum_{A=1}^n v^A_af^a_{\ b}v^A_dg^{bd}+ v^A_av^A_c\Gamma^a_{bd}f^dg^{bd}.
\end{equation}
Here $\Gamma$ denote the Christhoffel's symbols on $M$, ${}^M\nabla_{\frac{\partial}{\partial x^a}}\frac{\partial}{\partial x^b}=:\Gamma_{ab}^c\frac{\partial}{\partial x^c}$. Finally,
\begin{align*}
\frac12\nabla f(|dv|_{HS}^2) &= \frac12 \sum_{A=1}^n \nabla f \left(\mmetr{\left(\theta^A\circ dv\right)^{\sharp}}{\left(\theta^A\circ dv\right)^{\sharp}}\right)\\
&=\sum_{A=1}^n \mmetr{\nabla_{\nabla f }\left(\theta^A\circ dv\right)^{\sharp}}{\left(\theta^A\circ dv\right)^{\sharp}}.
\end{align*}
Since 
\[
\left(\theta^A\circ dv\right)^{\sharp}=v^A_ag^{ab}\frac{\partial}{\partial x^b},
\]
we get
\[
\nabla_{\nabla f }\left(\theta^A\circ dv\right)^{\sharp}=\left(f^c \left(v^A_ag^{ab}\right)_c + f^av^A_dg^{dc}\Gamma_{ac}^b\right)\frac{\partial}{\partial x^b},
\]
and
\begin{equation}\label{hess_3}
\frac12\nabla f(|dv|_{HS}^2) = \sum_{A=1}^n v^A_bf^av^A_dg^{dc}\Gamma_{ac}^b + f^av^A_{da}g^{bd}v^A_b + f^a v^A_dv^A_b \left(\frac{\partial}{\partial x^a}g^{db}\right). 
\end{equation}
Combining \eqref{hess_1}, \eqref{hess_2} and \eqref{hess_3} we get that \eqref{hess} is proved provided
\begin{equation}\label{hess_4}
2\sum_{A=1}^n v^A_bf^av^A_dg^{dc}\Gamma_{ac}^b + f^a v^A_dv^A_b \left(\frac{\partial}{\partial x^a}g^{db}\right) = 0.
\end{equation}
Now, since the Levi--Civita connection is compatible with the metric, 
\[
\frac{\partial}{\partial x^a}g_{bc} = g_{dc}\Gamma^d_{ba}+g_{db}\Gamma^d_{ca}
\]
and
\[
\left(\frac{\partial}{\partial x^a}g^{bc}\right)g_{cd} = - g^{bc}\left(\frac{\partial}{\partial x^a}g_{cd}\right),
\]
from which
\[
\frac{\partial}{\partial x^a}g^{bd}= - g^{bc}\Gamma_{ca}^d- g^{cd}\Gamma_{ca}^b
\]
and, in turn,
\begin{align*}
&\sum_{A=1}^nf^a\left[ 2v^A_bv^A_dg^{dc}\Gamma_{ac}^b + v^A_dv^A_b \left(\frac{\partial}{\partial x^a}g^{db}\right)\right] 
\\&=
\sum_{A=1}^nf^a\left[ 2v^A_bv^A_dg^{dc}\Gamma_{ac}^b - v^A_bv^A_dg^{bc}\Gamma_{ac}^d - v^A_bv^A_dg^{dc}\Gamma_{ac}^b \right] 
= 0.
\end{align*}
This latter proves \eqref{hess_4} and concludes the proof.
\end{proof}

We are now ready to prove the following vanishing result.

\begin{theorem}\label{f-SY}
Let $M^m,N^n$ be complete Riemannian manifolds, $M^m$ non--compact, and $f\in C^{\infty}(M)$. Assume that $\mricc_f\geq 0$ and $\nsect\leq 0$. Consider an  $f$--harmonic map $v:M\to N$ with finite $f$--energy $E_f(v)<+\infty$. 
\begin{itemize}
\item[(\textbf{I})] Assume that at least one of the following assumption is satisfied
\begin{itemize}
\item[(a)] there exists a constant $C>0$ such that $|f|\leq C$;
\item[(b)] $f$ is convex and the set of its critical points is unbounded;
\item[(c)] $\vol_f(M):=\int_Me^{-f}dV_M=+\infty$;
\item[(d)] there is a point $q_0\in M$ such that $\mricc_f|_{q_0}>0$;
\item[(e)] there is a point $q_1\in M$ such that $\mricc(X,X)|_{q_1}\neq 0$ for all $0\neq X\in T_{q_1}M$.
\end{itemize}
Then $v$ is constant.
\item[(\textbf{II})]If $\nsect<0$ then either $v$ is constant or the whole image $v(M)$ is contained in a geodesic of $N$.
\end{itemize}
\end{theorem}

\begin{remark}{\rm The case of $M$ compact is in \cite{L}, p. 367. See also \cite{SY-CH} for the non--compact harmonic case.}\end{remark}

We partially follow the proof given in \cite{PV} for solutions of generic Bochner--type inequalities. 

\begin{proof}
Set $\phi:=|dv|_{HS}$ and $G(v):=|Ddv|^2-|\nabla|dv||^2$. Notice that 
\[
\frac12\Delta_f\phi^2 = \phi\Delta_f\phi + |\nabla\phi|^2. 
\]
Since $v$ is $f$--harmonic, by the Bochner formula \eqref{f-bochner} and the Kato's inequality we get
\begin{align}\label{aus1}
\phi\Delta_f\phi \geq G(v) \geq 0.
\end{align}
Let $\rho\in C^{\infty}_c(M)$ to be chosen later. Then 
\[
\int_M\rho^2\phi\Delta_f\phi e^{-f}dV_M \geq \int_M\rho^2 G(v) e^{-f}dV_M \geq 0.
\]
Stokes' Theorem and Young's inequality yield
\begin{align*}
&\int_M\rho^2\phi\Delta_f\phi e^{-f}dV_M\\
&= - \int_M2\rho\phi\mmetr{\nabla\rho}{\nabla\phi}e^{-f}dV_M
- \int_M\rho^2|\nabla\phi|^2e^{-f}dV_M\\
&\leq \epsilon^{-1} \int_M \phi^2|\nabla\rho|^2e^{-f}dV_M - (1-\epsilon) \int_M\rho^2|\nabla\phi|^2e^{-f}dV_M
\end{align*}
for any $0<\epsilon<1$. From this latter we get
\begin{align}\label{caccioppoli}
0&\leq (1-\epsilon) \int_M\rho^2|\nabla\phi|^2e^{-f}dV_M + \int_M\rho^2 G(v) e^{-f}dV_M\\
&\leq \epsilon^{-1} \int_M \phi^2|\nabla\rho|^2e^{-f}dV_M\nonumber.
\end{align}
Choose the smooth cut--off $\rho=\rho_R$ s.t. $\rho\leq 1$ on $M$, $\rho|_{B_R}\equiv 1$, $\rho|_{M\setminus B_{2R}}\equiv 0$ and $|\nabla\rho|\leq 2/R$. Replacing $\rho=\rho_R$ in \eqref{caccioppoli} we obtain
\begin{align}\label{caccioppoli_R}
0&\leq (1-\epsilon) \int_{B_R}|\nabla\phi|^2e^{-f}dV_M + \int_{B_R} G(v) e^{-f}dV_M\\
&\leq \frac{4\epsilon^{-1}}{R^2}\int_{B_{2R}} \phi^2e^{-f}dV_M\nonumber.
\end{align}
By the assumption $E_f(v)<+\infty$ we can let $R\to\infty$ applying monotone convergence at RHS to get $\nabla\phi\equiv 0$, i.e. $\phi=|dv|\equiv const.$, and 
\begin{equation}\label{totgeo}
0\equiv G(v)=|Ddv|^2-|\nabla\phi|^2=|Ddv|^2.
\end{equation}
Suppose $|dv|\equiv C>0$. Then the finiteness of the $f$--energy of $v$ gives that $\vol_f(M)<+\infty$. If either $|f|$ is uniformly bounded or $f$ is convex and the set of its critical points is unbounded, then Theorem 1.3 and $5.3$ in \cite{WW-JDG} implies that $M$ has at least linear $f$--volume growth, giving a contradiction. 

In general we have $Ddv\equiv 0$, i.e. $v:M\to N$ is totally geodesic, which in turn gives that $v$ is harmonic, i.e. $\tau v=0$ and 
\begin{equation}\label{if}
i_{\nabla f}dv=dv(\nabla f)=\tau_fv-\tau v\equiv 0.
\end{equation}
Accordingly, the Bochner formula \eqref{f-bochner} reads
\begin{align}\label{f-bochner-v}
0 
&= \sum_{i=1}^m\nmetr{dv(\mricc_f(E_i,\cdot)^{\sharp})}{dv (E_i)}\\ 
&- \sum_{i,j=1}^m\nmetr{{}^N\operatorname{Riem}(dv (E_i),dv (E_j))dv (E_j)}{dv (E_i)} \nonumber, 
\end{align}
and by the curvature sign assumptions both
\begin{equation}\label{ric_f}
\sum_{i=1}^m\nmetr{dv(\mricc(E_i,\cdot)^{\sharp})}{dv (E_i)}+ \sum_{i=1}^m\nmetr{dv({}^M\nabla_{E_i}\nabla f)}{dv (E_i)}=0
\end{equation}
and
\[
\sum_{i,j=1}^m\nmetr{{}^N\operatorname{Riem}(dv (E_i),dv (E_j))dv (E_j)}{dv (E_i)}=0. 
\]
First, suppose that $\nsect<0$, then $dv(E_i)\parallel dv(E_j)$ for all $i,j=1,\dots,n$ and we conclude that $v(M)$ must be contained in a geodesic of $N$.

On the other hand, suppose that $\mricc_f|_{q_0}>0$ at some point $q_0\in M$. Then necessarily $dv(q)=0$ which gives $dv\equiv0$.

Moreover, since
\[
0= Ddv(X,Y)=(D_Ydv)(X)={}^N\nabla_{dv(Y)}dv(X)-dv({}^M\nabla_YX)
\]
for all $X,Y$ vector fields on $M$, \eqref{if} implies
\begin{equation*}
\nmetr{dv({}^M\nabla_{E_i}\nabla f)}{dv (E_i)}=\nmetr{{}^N\nabla_{dv(E_i)}dv(\nabla f)}{dv(E_i)}=0
\end{equation*}
for each $i=1,\dots,m$. Since $\mricc_f\geq 0$, \eqref{ric_f} in particular gives
\[
\nmetr{dv(\mricc(E_i,\cdot)^{\sharp})}{dv (E_i)}=0
\]
for each $i=1,\dots,m$. Hence, if there exists some point $q_1\in M$ such that $\mricc(X,X)|_{q_1}\neq 0$ for all vector $0\neq X\in T_{q_1}M$, then $v$ is once again necessarily constant.
\end{proof}

\begin{remark}\label{linear}{\rm Theorem \ref{f-SY} holds also in the more general assumption
\[
\int_{B_R}|dv|^2e^{-f}dV_M = o(R),
\]
instead of $E_f(v)<+\infty$.
}\end{remark}

\begin{remark}\label{ric_diff}{\rm One could combine Proposition \ref{prop_conf} or Proposition \ref{prop_warp} above with the vanishing result in the harmonic case to obtain directly a vanishing result for $f$--harmonic maps. Nevertheless in this case the assumptions on ${}^M\ricc$ are more involved and less natural in view of the applications to smooth metric measure spaces and gradient Ricci solitons. Namely, one has that
\[
{}^{\widetilde M}\ricc= {}^M\ricc + \left[\Hess f + \frac{df\otimes df}{m-2}\right]+\frac1{m-2}\left[\Delta f- |df|^2\right]g_{M}
\]
where $\widetilde M = (M,e^{-\frac{2f}{m-2}}g_{M})$, while 
\begin{equation*}
{}^{\bar{M}}\ricc(X,X)= {}^M\ricc(X,X) + \Hess f(X,X) - g_{M}(X,\nabla f)^2,
\end{equation*}
for $X\in TM<T\bar M$, and
\[
{}^{\bar{M}}\ricc (\nu,\nu)= \Delta f - |df|^2 = \Delta_ff,
\]
for $\nu\in T\mathbb T<T\bar M$.
}\end{remark}

\begin{remark}\label{nontriv}{\rm
The statement of Theorem \ref{f-SY} is non--trivial, as shown in the following example. 

Let $\ttt^n$ be the standard flat $n$--torus $\rr^n/\mathbb{Z}^n$. For the easiness of notation, here we parameterize the torus as $\ttt^n=[-1,1]^n$ with the usual identifications on the boundaries. Consider a map $v:\rr^2\times\ttt=M\to N=\ttt^3$ and a function $f:\rr^2\times\ttt=M\to\rr$ given by $v(x,y,[z]):=([0],[0],[z])$ and $f(x,y,[z]):=\frac{x^2+y^2}{2}$. Then $^{M}\ricc_f=dx\otimes dx+dy\otimes dy\geq 0$, and since $dv=dz\otimes\frac{\partial}{\partial z}$ is parallel, we have that $\tau v =0$. Moreover $df = xdx+ydy$ implies $\tau_fv=dv(\nabla f)=0$. As a consequence $v$ is a non--constant $f$--harmonic map (notably in a non--trivial homotopy class of maps), and in fact none of the assumptions (a) to (e) of Theorem \ref{f-SY} is satisfied. In particular, (a) $f$ is upper unbouded, (b) $f$ is convex but with a bounded set of critical points, (c) the $f$--volume of $M$ satisfies
\[
\vol_f(\rr^2\times\ttt)=\int_{-1}^1\int_{\rr^2}e^{-\frac{x^2+y^2}2}dxdydz =4\pi<+\infty,
\] 
(d) $\mricc_f(\frac{\partial}{\partial z},\frac{\partial}{\partial z})\equiv0$ and (e) $\mricc\equiv 0$.
}\end{remark}
\begin{remark}{\rm
A similar example shows that the weaker assumption suggested in Remark \ref{linear} is sharp. In particular let $M$, $N$ and $v$ be as in Remark \ref{nontriv} but now choose $f(x,y,[z])=\frac{x^2}2$. With these choices, $\mricc_f(M)\geq 0$ and assumption (b) of Theorem \ref{f-SY} is satisfied, but the $f$--volumes of geodesic balls have exactly linear growth and in fact $v$ is once again a non--constant $f$--harmonic map.
}\end{remark}
Combining Theorem \ref{f-burstall} and Theorem \ref{f-SY} we get the following result, which corresponds to the cases (\textbf{I}) and (\textbf{III}) of Theorem \ref{th_main}. 

\begin{theorem}\label{f-homot}
Let $M$ be a complete non--compact Riemannian manifold and $N$ a compact Riemannian manifold with $\nsect\leq0$. Let $f\in C^\infty(M)$ and suppose $\mricc_f\geq 0$. Consider a continuous map $u:M\to N$ with finite $f$--energy $E_f(u)< +\infty$. Then $u$ is homotopic to a constant provided at least one of the assumptions \textit{(a)} to \textit{(e)} of Theorem \ref{f-SY} is satisfied.

On the other hand, if we assume that $\nsect<0$, then $u$ is homotopic either to a constant or to a totally geodesic map whose image is contained in a geodesic of $N$.
\end{theorem}

\section{Spectral assumptions}\label{SpectSect}

In this section we suppose that the underlying manifold $M$ satisfies 
\begin{equation*}
\mricc_f\geq- k^2(x)
\end{equation*}
for some function $k(x)$, provided 
\begin{equation}\label{f-spect-1}
\lambda_1(-\Delta_f-Hk^2)\geq 0
\end{equation} 
for some $H>1$. As pointed out in Remark \ref{rmk_rai}, this corresponds to ask that
\begin{equation*}
\int_M|\nabla\varphi|^2e^{-f}dV_M-\int_MHk^2\varphi^2e^{-f}dV_M\geq 0
\end{equation*}
for all $\varphi\in C^{\infty}_c(M)$.
On the other hand, it was observed in \cite{V,BPS} that \eqref{f-spect-1} is also equivalent to ask that
\[
\lambda_1\left(-\Delta-\left(\frac{1}2\Delta f - \frac14|\nabla f|^2+Hk^2\right)\right)\geq 0.
\]
Proceeding as in \cite{PV}, we can show that Theorem \ref{f-SY} and Theorem \ref{f-homot} hold in these more general assumptions. The result thus obtained, together with Theorem \ref{f-homot}, will complete the proof of Theorem \ref{th_main}. 

\begin{theorem}\label{f-SY_spectr}
Let $M^m,N^n$ be complete Riemannian manifolds, $M^m$ non--compact, and $f\in C^{\infty}(M)$. Assume that $\mricc_f\geq -k^2(x)$ for some function $k(x)\not\equiv 0$ satisfying \eqref{f-spect-1} and that $\nsect\leq 0$. Then the following hold
\begin{itemize}
\item [(i)] Any smooth $f$--harmonic map $v:M\to N$ with finite $f$--energy $E_f(v)<+\infty$ is constant.
\item [(ii)] If $N$ is compact, then any continuous map $u:M\to N$ with finite $f$--energy $E_f(u)< +\infty$ is homotopic to a constant.
\end{itemize}
\end{theorem}

\begin{proof}
First we note that, thanks to Theorem \ref{f-burstall}, (i) trivially implies (ii).

To prove (i), we define objects as in the proof of Theorem \ref{f-SY}. Choosing $\varphi=\rho\phi$, \eqref{ray} yields that
\begin{align*}
\int_Mk^2\rho^2\phi^2e^{-f}dV_M &\leq (1+\delta)H^{-1}\int_M \rho^2|\nabla\phi|^2 e^{-f}dV_M \\
&+ (1+\delta^{-1})H^{-1}\int_M \phi^2|\nabla\rho|^2 e^{-f}dV_M,
\end{align*}
for any fixed $\delta>0$. Proceeding as in the proof of Theorem \ref{f-SY} we note that inequalities \eqref{aus1} and \eqref{caccioppoli} are respectively replaced by
\begin{align*}
\phi\Delta_f\phi + k^2\phi^2 \geq G(v) \geq 0.
\end{align*}
and
\begin{align*}
0&\leq (1-\epsilon-H^{-1}(1+\delta)) \int_M\rho^2|\nabla\phi|^2e^{-f}dV_M + \int_M\rho^2 G(v) e^{-f}dV_M\\
&\leq (\epsilon^{-1}+H^{-1}(1+\delta^{-1})) \int_M \phi^2|\nabla\rho|^2e^{-f}dV_M\nonumber,
\end{align*}
up to choose $\delta$ small enough. From now on, the proof is the same as for the case $k\equiv0$. We have only to remark that the spectral assumption implies that $M$ has infinite weighted volume, as observed in Remark 1.8 (b) in \cite{PRRS}. In fact, by contradiction, if $\vol_f(M)<+\infty$ (more generally if it has at most linear growth), then in \eqref{ray} we can choose $\varphi=\varphi_R$ to be a family of cut--offs such that $\varphi_R\leq 1$, $\varphi_R|_{B_R}\equiv1$, $\varphi_R|_{M\setminus B_{2R}}\equiv 0$ and $|\nabla\varphi_R|\leq 2/R$, and letting $R\to\infty$ we obtain that $k\equiv 0$.
\end{proof}

\begin{remark}\label{rmk_kato}{\rm
With respect to the standard $f=0$ case, here we need $H>1$ because no refined Kato inequality is given for $f$--harmonic maps. We recall that for a smooth harmonic map $v$, the refined Kato inequality is the relation, \cite{Br,CGH}
\[
|Ddv|^2-|\nabla|dv||^2\geq K |\nabla|dv||^2 
\] 
holding with $K=1/(m-1)$. As a matter of fact, it turns out that such a constant $K>0$ can not exist for general $f$--harmonic maps. To see this, consider functions $v,f:\rr^3\to\rr$ given by 
\[
v(x,y,z):=\int_0^x e^{-\frac{t^2}2}dt,\qquad f(x,y,z)=-\frac{x^2}2.
\]
We have that 
\[
\Delta v = \frac{\partial^2 v}{\partial x^2} = -xe^{-\frac{x^2}{2}} 
\] 
and
\[
\left\langle \nabla f, \nabla v\right\rangle = \left\langle-x \frac{\partial}{\partial x}, e^{-\frac{x^2}{2}}\frac{\partial}{\partial x}\right\rangle = - xe^{-\frac{x^2}{2}},
\]
so that 
\[
\tau_f v =\Delta v-\left\langle \nabla f,\nabla v\right\rangle=0.
\]
On the other hand $d|\nabla v|= -xe^{-\frac{x^2}{2}}dx$ and $Ddv=-xe^{-\frac{x^2}{2}}dx^2$, which implies
\[
|Ddv|^2 \equiv |\nabla|dv||^2.
\]
}\end{remark}

\section{Geometric context and applications}\label{main}

The importance of studying topological rigidity properties of smooth metric measure spaces under curvature restrictions arises from the need to understand the topology of gradient Ricci solitons. 

Concerning the shrinking side, i.e. $\ricc_f=\lambda g$ for some constant $\lambda>0$, we mention that it is not difficult to see that the full conclusion of the classical Myers--Bonnet theorem cannot be extended to $\ricc_f$. Indeed the Gaussian space $(\mathbb{R}^m, \left\langle \,,\,\right\rangle_{can}, e^{-|x|^2/2}dV_{\mathbb{R}^m})$ is a non--compact, complete smooth metric measure space with $\ricc_f=1>0$. In order to recover compactness we have to impose, besides the positive constant lower bound on $\ricc_f$, further conditions on the growth of $f$ or on its gradient, see \cite{FLGR-Cmp,ELNM}. Nevertheless, as initially investigated in works of M. Fern\'andez--L\'opez and E. Garc\'ia--R\'io in the compact case, \cite{FLGR-Cmp}, and later in the complete non--compact case by W. Wylie, \cite{W}, a close relationship between $\ricc_f$ and the fundamental group of a smooth metric measure space still exists. Namely, Myers--type results in this context establish the finiteness of the fundamental group if $\ricc_f\geq c^2>0$ (and in particular for shrinking gradient Ricci solitons). Let us also mention that very recently in \cite{PRRS} these results were extended in the direction of the classical Ambrose theorem to complete smooth metric measure spaces $(M^m, g_{M}, e^{-f}dV_M)$ such that, fixed a point $o\in M$, for every unit speed geodesic $\gamma$ issuing from $\gamma(0)=o$ we have
\begin{equation*}
(i)\,\,\mricc_f(\dot{\gamma},\dot{\gamma})\geq\mu\circ\gamma+\mmetr{\nabla g\circ\gamma}{\dot{\gamma}},\qquad(ii)\,\,
\int_{0}^{+\infty}\mu\circ\gamma(t)dt=+\infty,
\end{equation*}
fore some functions $\mu\geq0$ and $g$ bounded. For more details see Theorem 9.1 in \cite{PRRS}.

We now come to analyze the topology of smooth metric measure spaces $(M^m,$ $g_{M}, e^{-f}dV_M)$ with $\mricc_f\geq0$. Altough we are interested in particular to the study of topological properties ``in the small'' we mention that very recently some results regarding the topology at infinity of smooth metric measure spaces with $\mricc_f\geq0$ and in particular of gradient steady Ricci solitons were obtained by O. Munteanu and J. Wang in \cite{MW1}. In particular, they prove that steady gradient Ricci solitons are either connected at infinity or they are isometric to a Ricci flat cylinder.

The following proposition encloses some known topological results in case $\mricc_f\geq 0$ which are particularly relevant to our investigation. These are obtained in \cite{Li, WW-JDG, Ya} and adapting to weighted setting a result in \cite{So}. Note that, as we observed in Remark \ref{Wu..}, the potential function of a non--trivial steady gradient Ricci soliton must be unbounded. Thus, except for the first part of point (iv), none of these results apply to non--trivial steady gradient Ricci solitons. This fact attaches importance to our Corollary \ref{coro_steady}.
\begin{proposition}\label{th_knowntopo}
Let $(M^m, g_M, e^{-f}dV_M)$ be an $m$--dimensional complete smooth metric measure space, then the following hold: 
\begin{enumerate}
\item[(i)] if $M$ is compact, $\mricc_f\geq0$ and $\mricc_f>0$ at one point then $|\pi_1(M)|<\infty$;
\item[(ii)] if $M$ is non--compact, $\mricc_f\geq0$, and $|f|$ is bounded, then $M$ either satisfies the loops to infinity property or has a double covering which splits. In particular if $\mricc_f>0$ then $M$ satisfies the loops to infinity property;
\item[(iii)] if $M$ is non--compact with $\mricc_f\geq 0$, $|f|$ is bounded and $D\subset M$ is a precompact set with simply connected boundary, then $\pi_1(D)$ can only contain elements of order $2$.
\item[(iv)] if $M$ is non--compact and $\mricc_f\geq0$, and either $f$ is a convex function and attains its minimum or $|f|$ is bounded, then $b_1(M)\leq m$.
\end{enumerate}
\end{proposition}

As a first application of their vanishing result for the harmonic representative, R. Schoen and S.--T. Yau, in \cite{SY-CH}, studied the topology of manifolds with non--negative Ricci curvature. We can now naturally generalize their work to the case of manifolds with non--negative Bakry--\'Emery Ricci curvature. Contrary to all the results contained in Proposition \ref{th_knowntopo}, the technique used here can be extended, under the validity of the spectral assumption \eqref{f-spect} on $\Delta_f$, also to the more general case $\mricc_f\geq -k^2$. In particular as a consequence of Theorem \ref{th_main} we obtain Corollary \ref{coro_main}, which predicts information on the topology of compact domains with simply connected boundary in a complete manifold $M$. We would like to underline that, while Corollary \ref{coro_main} (i) is a simple generalization of the non--weighted results, the conclusion (ii) is new. In fact, in case $f\equiv 0$, condition (a) of Theorem \ref{th_main} is automatically satisfied, so that, since we can apply part (\textbf{I}) (or part (\textbf{II})) of Theorem \ref{th_main}, (ii) is trivially contained in (i). Accordingly, in order to prove Corollary \ref{coro_main} in the weighted setting, some further work is necessary. In particular we need the following result.

\begin{theorem}\label{th_B(ii)}
Let $M$ and $N$ be complete Riemannian manifolds, $u\in C^2(M,N)$ such that $Ddu=0$ and $\operatorname{rk}(u)\equiv1$. Then $u(M)\subseteq\gamma$, $\gamma$ geodesic of $N$ and
\begin{enumerate}
\item[(1)] if $\gamma$ is closed $u_{\sharp}(\pi_1(M, x_0))\leq\left\langle \left[\gamma\right]_{u(x_0)}\right\rangle\leq\pi_1(N,u(x_0))$;
\item[(2)] if $\gamma$ is not closed then $u_{\sharp}(\pi_1(M, x_0))=\mathbf{1}_{\pi_1(N, u(x_0))}$.
\end{enumerate}
\end{theorem}

\begin{proof}
Since $u$ is a totally geodesic map of rank $1$, by standard arguments we know that there exists a geodesic $\gamma$ of $N$ such that $u(M)\subset\gamma$. Without loss of generality, we can take a constant speed parametrization $\gamma:\rr\to N$. Fix an element $\mathfrak g\in\pi_1(M,x_0)$. In case $\gamma$ is closed, hence periodic so that up to a linear change of variable it can be reparametrized as $\gamma:\ttt\to N$, we want to show that
\begin{equation}\label{concl_2}
[u\circ\sigma]_{u(x_0)}=[\gamma]^l_{u(x_0)}
\end{equation}
for some $l\in\mathbb{Z}$ and for some (hence any) continuous loop $\sigma:[0,1]\to M$ based at $x_0$, i.e. $\sigma(0)=\sigma(1)=x_0$, such that $[\sigma]_{x_0}=\mathfrak g$. On the other hand, in case $\gamma:\rr\to N$ is not closed, i.e. non--periodic, the thesis corresponds to prove that
\begin{equation}\label{concl_1}
[u\circ\sigma]_{u(x_0)}=\mathbf{1}_{\pi_1(N,u(x_0))}.
\end{equation}
First, we show that we can choose the continuous loop $\sigma:[0,1]\to M$ based at $x_0$, $[\sigma]_{x_0}=\mathfrak g$, in such a way that $\sigma|_{(0,1)}$ is a constant speed geodesic. To this end, let $\tilde M$ be the universal cover of $M$ and $P_M:\tilde M\to M$ the covering projection. Fix an element $\tilde x_0$ in the fiber $P_M^{-1}(x_0)$. Then we can lift $\sigma$ to a continuous path $\tilde\sigma:[0,1]\to\tilde M$ which satisfies 
\[
P_M\circ\tilde\sigma(t) =\sigma(t),\qquad \tilde\sigma(0)=\tilde x_0,\qquad \tilde\sigma(1)=\tilde x_0'
\]
for all $t\in[0,1]$ and for some $\tilde x_0'\in P_M^{-1}(x_0)$ (possibly $\tilde x_0=\tilde x_0'$). Consider a constant speed (possibly constant) geodesic $\tilde\nu:[0,1]\to\tilde M$ joining $\tilde x_0$ to $\tilde x_0'$. Since $\tilde M$ is simply connected there exists a homotopy $\tilde H$ relative to $\{\tilde x_0,\tilde x_0'\}$ deforming $\tilde\sigma$ into $\tilde\nu$, i.e. $\tilde H\in C^{0}([0,1]^2,\tilde M)$ and $\tilde H(0,t)=\tilde\sigma(t)$, $\tilde H(1,t)=\tilde\nu(t)$, $\tilde H(s,0)=\tilde x_0$ and $\tilde H(s,1)=\tilde x_0'$ for all $s,t\in [0,1]$. Hence we can project $\tilde H$ to the homotopy $H:=P_M\circ\tilde H\in C^{0}([0,1]^2,M)$, so that the loop $\nu$ based at $x_0$ defined by $\nu(t):=H(1,t)=P_M(\tilde\nu(t))$ satisfies $[\nu]_{x_0}=[\sigma]_{x_0}=\mathfrak g$ and, since covering projection maps are local isometries, $\nu|_{(0,1)}$ is a geodesic.

Now, if $u\circ\nu$ is constant, then clearly $[u\circ\nu]_{u(x_0)}=\mathbf{1}_{\pi_1(N,u(x_0))}$, so that $u_\sharp([\nu]_x)=\mathbf{1}_{\pi_1(N,u(x))}$. On the other hand, suppose that $u\circ\nu\in C^{0}([0,1],N)$ is not constant. Then, since $Ddu=0$, $u\circ\nu$ is a non--trivial constant speed geodesic arc in $N$, i.e. ${}^N\nabla_{du(\dot\nu)}du(\dot\nu)\equiv 0$, and $u\circ\nu\subset\gamma$. Since $\nu(0)=\nu(1)$, $\nu$ can be seen as a continuous function on $\ttt=[0,1]/\sim$ smooth in $(0,1)$. Moreover, since $Ddu=0$ and $\operatorname{rk}(u)\equiv 1$, we have that 
\[
du(\dot\nu(1)) \parallel du(\dot\nu(0))\qquad\textrm{and}\qquad |du(\dot\nu(1))|=|du(\dot\nu(0))| \neq 0
\]
in $T_{u(x_0)}N$, and, since $u\circ\nu|_{(0,1)}$ is a non--trivial geodesic of $N$, necessarily it is
\[
du(\dot\nu(1)) = du(\dot\nu(0)).
\]
Then both $\gamma$ and $u\circ\nu$ are non--trivial closed geodesic of $N$, and $u\circ\nu=\gamma^l$ (i.e. $\gamma$ covered $l$ times) for some $l\in\mathbb{Z}$. In particular we get $u_\sharp([\nu]_x)= [\gamma]^l_{u(x)}$.

Since we have arbitrarily chosen the element $\mathfrak g\in \pi_1(M,x_0)$, this permits to conclude.
\end{proof}

As a consequence of Theorem \ref{th_B(ii)} and Theorem \ref{th_main} we obtain the desired topological result.

\begin{proof}[Proof (of Corollary \ref{coro_main})]
We give just an idea of the proof. For the details see the proof of Theorem 6.21 in \cite{PRS-Book}.

Consider a homomorphism $\sigma\in\operatorname{Hom}(\pi_1(D),\pi_1(N))$. Since $N$ is $K(\pi,1)$, according to the theory of aspherical spaces there exists a map $\hat u:D\to N$ such that $\sigma=\alpha\circ\hat u_{\sharp}$ for some automorphism $\alpha\in \operatorname{Aut}(\pi_1(N))$. Since $\partial D$ is simply connected, $\hat u$ can be extended to a map $u:M\to N$ such that $u|_{M\setminus D'}$ is constant for some compact set $D\subset\subset D'\subset\subset M$. Then $u$ has finite $f$--energy and we can apply Theorem \ref{th_main}. First, if the assumptions in (\textbf{I}) or (\textbf{II}) are satisfied, we deduce that $u$ is homotopic to a constant and accordingly the homomorphism $\sigma$ is trivial. On the other hand, if $\nsect<0$, then Theorem \ref{th_main} says that $u$ is a totally geodesic map of $\operatorname{rk}(u)\equiv 1$ so that $u(M)$ is contained in some geodesic $\gamma$ of $N$. Hence, an application of Theorem \ref{th_B(ii)} yields that, for some $x_0$ in $D$, 
\begin{align*}
\hat u_{\sharp}(\pi_1(D,x_0))&= u_{\sharp}(\pi_1(D,x_0))< u_{\sharp}(\pi_1(M,x_0))\\
&< \left\langle \left[\gamma\right]_{u(x_0)}\right\rangle<\pi_1(N,u(x_0)).
\end{align*}
\end{proof}

\begin{remark}\label{Remark 5.3}
\rm{
When $M$ is compact and (d) holds or $M$ is complete non--compact and (a) is satisfied, the conclusion of Corollary \ref{coro_main} (i) follows from Proposition \ref{th_knowntopo} (i) and (iii). Indeed, let $\rho:\pi_1(M)\to\pi_1(N)$ be  a non--trivial homomorphism of $\pi_1(M)$ into the fundamental group of a compact manifold $N$ with $\nsect\leq0$. Then every $\mathfrak g \in \pi_1(M)$, and hence $\rho(\mathfrak g)$, must have finite order. On the other hand by Cartan theorem all non--trivial elements of the fundamental group of a complete Riemannian manifold of non--positive curvature have infinite order. Contradiction.

}\end{remark}

Let us now consider the case where $(M^m, g_{M}, e^{-f}dV_M)$ supports an expanding gradient Ricci soliton structure. Very recently in \cite{MW2} it has been shown that a non--trivial expanding Ricci soliton must be connected at infinity provided its scalar curvature satisfies a suitable lower bound. Starting from Corollary \ref{coro_main} we can go a step further in the understanding of the topology of this class of manifolds. Indeed, in \cite{MW2} the following lemma is proved. Letting $^MS$ denote the scalar curvature of $M$, recall that by the scalar curvature estimates proved for expanding Ricci solitons, $m\lambda\leq \inf_M\,{}^MS\leq 0$ and ${}^MS(x)>m\lambda$, unless $M$ is Einstein and the soliton is trivial (i.e. $f$ is constant); see \cite{PRiS}.
\begin{lemma}[Lemma 5.3 in \cite{MW2}]\label{LemmaMuntWang}
Let $(M^m, g_{M}, \nabla f)$ be a complete non--trivial expanding gradient Ricci soliton. Define $\rho:={}^MS-m\lambda$. Then $\rho>0$ and 
\begin{equation*}
\int_M\rho\varphi^2e^{-f}dV_M\leq\int_M|\nabla\varphi|^2e^{-f}dV_M,
\end{equation*}
for any $\varphi\in C_0^{\infty}$. In particular $\lambda_1(-\Delta_f-\rho)\geq 0$.
\end{lemma} 

\begin{remark}
\rm{As a consequence of Lemma \ref{LemmaMuntWang} and Corollary 1.7 in \cite{PRRS} we observe incidentally that for any complete non--trivial expanding Ricci soliton with scalar curvature such that ${}^MS-m\lambda>c>0$ the quantity $\int_{B_r}|\nabla f|^pe^{-f}dV_M$ has to grow at least quadratically, for any $p>1$. }
\end{remark}

The lower bound for the bottom of the spectrum of $\Delta_f$ obtained in Lemma \ref{LemmaMuntWang} permits to obtain conditions on the scalar curvature of an expanding gradient Ricci soliton in order to guarantee condition \eqref{f-spect} for some $H>1$, where $k\not\equiv 0$, and hence to apply Corollary \ref{coro_main} under the assumptions in (\textbf{II}). 
\begin{proof} [Proof of Theorem \ref{th_exp}] 
If $\inf_M {}^MS>(m-1)\lambda$, then the proof is easy. In fact in this case one can choose $H>1$ such that
\begin{equation}\label{spect_MuntWang}
\int_M ({}^MS-m\lambda)\varphi^2 e^{-f}dV_M\geq-\int_MH\lambda\varphi^2e^{-f}dV_M
\end{equation}
for any $\varphi\in C_0^{\infty}$. By Lemma \ref{LemmaMuntWang} this permits to deduce \eqref{ray} with $k^2=-\lambda$.

In general, let us remark that \eqref{spect_MuntWang} is used in the proof of Corollary \ref{coro_main} (see also Theorem \ref{f-SY_spectr}) with the choices $\varphi^2=\phi^2\rho^2$, where $\phi^2=|dv|^2$ is the energy density of a finite $f$-energy $f$-harmonic map, and $\rho=\rho_R$ are a sequence of cut-off functions, increasing in $R$ and converging pointwise to 1 as $R\to\infty$.
Observe that since ${}^MS>(m-1)\lambda$, $\lambda$ is constant and $\int_M\phi^2e^{-f}dV_M<+\infty$, then there exists an $H'>1$ such that
\begin{equation*}
\int_M ({}^MS-m\lambda)\phi^2 e^{-f}dV_M\geq-\int_MH'\lambda\phi^2e^{-f}dV_M.
\end{equation*}
Fix $1<H<H'$, then by monotone convergence, there exists $\bar R>0$ such that 
\begin{equation*}
\int_M ({}^MS-m\lambda)\rho_R^2\phi^2 e^{-f}dV_M\geq-\int_MH\lambda\rho_R^2\phi^2e^{-f}dV_M
\end{equation*}
for any $R>\bar R$. Accordingly, with this choice for $H>1$, we can repeat the proof of Corollary \ref{coro_main} to obtain Theorem \ref{th_exp}.
\end{proof}

\section*{Acknowledgements}
We thank the anonymous referee for useful suggestions that greatly helped to improve the exposition.
\bibliographystyle{amsplain}
\bibliography{Bibliography_f-harm}                                     
\end{document}